\documentclass[10pt,reqno]{amsart} 
\usepackage{amsmath} 
\usepackage{amssymb} 
\usepackage{amsthm} 
\usepackage{mathtools}
\usepackage{comment}
\usepackage{tikz}
\usepackage[colorlinks]{hyperref}
\usepackage{float}
\usepackage{enumitem}

\usepackage{tablefootnote}
\usepackage{array}

\usepackage[normalem]{ulem}

\usepackage{tikz}
\usetikzlibrary{arrows.meta, positioning}

\numberwithin{equation}{section}

\newcommand{\sub}{\subseteq}

\newcommand{\R}{\mathbb{R}}

\newcommand{\N}{\mathbb{N}}
\newcommand{\eps}{\varepsilon}

\numberwithin{chap}{section}
\newtheorem{thm}{Theorem}
\numberwithin{thm}{section}
\newtheorem{conj}[thm]{Conjecture}
\newtheorem{prop}[thm]{Proposition}
\newtheorem{defn}[thm]{Definition}
\newtheorem{lem}[thm]{Lemma}

\newtheorem{cor}[thm]{Corollary}

\DeclarePairedDelimiter{\norm}{\lVert}{\rVert}

\makeatletter
\let\oldnorm\norm
\def\norm{\@ifstar{\oldnorm}{\oldnorm*}}

\makeatother

\begin{document}

\pagestyle{myheadings} \thispagestyle{empty} \markright{}
\title{Decoupling for degenerate hypersurfaces}

\author{Jianhui Li and Tongou Yang}
\address[Jianhui Li]{Department of Mathematics, Northwestern University\\
Evanston, IL 60208, United States
}
\email{jianhui.li@northwestern.edu}

\address[Tongou Yang]{Department of Mathematics, University of California\\
Los Angeles, CA 90095, United States}
\email{tongouyang@math.ucla.edu}

\date{}

\begin{abstract}
We utilise the two principles of decoupling introduced in \cite{LiYang2024} to prove the following conditional result: assuming uniform decoupling for graphs of polynomials in all dimensions with identically zero Gaussian curvature, we can prove decoupling for all smooth hypersurfaces in all dimensions. Moreover, we are able to prove (unconditional) decoupling for all smooth hypersurfaces in $\R^4$ and graphs of homogeneous polynomials in $\R^5$.
\end{abstract}

\maketitle

\section{Introduction}
Fourier decoupling inequalities, first introduced by Wolff in \cite{Wolff2000}, are well established for hypersurfaces with non-zero principal curvatures; see \cite{BD2015, BD2017} and also the textbook \cite{Demeter2020}. As with many other problems in harmonic analysis, the absence of curvature assumptions makes the analysis significantly more subtle.

Building on previous efforts \cite{BDK2019, LiYang, LiYang2023, LiYang2024, LiYang2025, Kemp1, Kemp2, Kemp2024} to develop decoupling inequalities for degenerate hypersurfaces, one of the main results in this paper (see Corollary \ref{cor:conj_holds}) is to reduce the decoupling problem for all smooth hypersurfaces in $\R^{n+1}$ to proving uniform decoupling inequalities for graphs of polynomials $\phi : \R^n \to \R$ with identically zero Gaussian curvature, i.e., $\det D^2\phi \equiv 0$.

The classification of such degenerate polynomials is a challenging algebraic problem in general. However, in dimension $n=3$, a complete classification is available thanks to \cite{SingularHessian2004}. According to their result, any such function $\phi$ either depends on at most two variables or takes a special form that we refer to as \textit{near-affine}. See Lemma \ref{lem:zero_Hessian_paper} below for a precise statement. We then establish decoupling inequalities for near-affine hypersurfaces by using the techniques developed in \cite{LiYang2024}. This leads to an unconditional decoupling result for smooth hypersurfaces in $\R^4$ (see Theorem \ref{thm:almost_affine}), where our method is extended to all dimensions.

Furthermore, applying a similar argument to \cite[Section 12]{LiYang2025}, we are able to prove a uniform decoupling result for homogeneous polynomials of 4 variables (see Corollary \ref{cor:H(4)}).

\subsection{Formulation of decoupling}
We first formulate the decoupling inequalities concerned with this paper. We also recommend that the reader refer to \cite{LiYang2024} for more general formulations of decoupling.

\begin{defn}\label{defn:decoupling}
    Given a compact subset $S\sub \R^n$ and a finite collection $\mathcal R$ of boundedly overlapping parallelograms\footnote{See Section \ref{sec:notation} for the precise definitions.} $R\sub \R^n$. For Lebesgue exponents $p,q\in [2,\infty]$, we define the $\ell^q(L^p)$-decoupling constant $\mathrm{Dec}(S,\mathcal R,p,q)$ to be the smallest constant $\mathrm{Dec}$ such that
    \begin{equation}\label{eqn:defn_decoupling}
        \norm{\sum_{R}f_R}_{L^p(\R^n)} \leq \mathrm{Dec}\,\, (\#\mathcal R)^{\frac{1}{2}-\frac{1}{q}}\norm{\norm{f_R}_{L^p(\R^n)}}_{\ell^q(R\in \mathcal R)}
    \end{equation}
    for all smooth test function $f_R$ Fourier supported on $R\cap S$.
    
    Given a subset $S\sub \R^n$, we say $S$ can be $\ell^q(L^p)$-decoupled into the parallelograms $R\in \mathcal R$ at the cost of $K$, if $S\sub \cup \mathcal R$ and $\mathrm{Dec}(S,\mathcal R,p,q)\le K$.
\end{defn}

Since we often deal with decoupling for graphs of functions, we introduce the following notation.
\begin{defn}\label{defn:graphical_decoupling}
    Let $\Omega\sub \R^k$ be an open set, $\phi:\Omega\to \R$ be $C^2$. We say that a parallelogram $\Omega_0\sub \Omega$ can be $\phi$-$\ell^q(L^p)$ decoupled into parallelograms $\omega$ at scale $\delta$ (at cost $C$), if $N_\delta^\phi(\Omega_0)$ can be $\ell^q(L^p)$ decoupled into parallelograms equivalent to $N_\delta^\phi(\omega)$ (at cost $C$). 
\end{defn}

We now state the following variant of the celebrated decoupling theorem of Bourgain and Demeter \cite{BD2015,BD2017}, which serves as the most fundamental ingredient in this article. See also \cite[Proposition 5.21]{LiYang2024}.
\begin{thm}[Bourgain-Demeter]\label{thm:Bourgain_Demeter}
Let $K\ge 1$ and $\phi(x):[-1,1]^{n-1}\to \R$ be a $C^{3}$ function with $\inf |\det D^2 \phi|\ge K^{-1}$. For every $0<\delta<1$, denote by $\mathcal T_\delta$ a tiling of $[-1,1]^{n-1}$ by cubes $T$ of side length $\delta^{1/2}$.

Then for $2\le p\le \frac{2(n+1)}{n-1}$, $[-1,1]^{n-1}$ can be $\phi$-$\ell^p(L^p)$ decoupled into $\mathcal T_\delta$ at the cost of $C_\eps K^{C_n\eps}\delta^{-\eps}$ for every $\eps>0$, in the sense of Definition \ref{defn:decoupling}, where the constant $C_n$ depends only on $n$, and $C_\eps$ depends only on $\eps,p,\norm{\phi}_{C^3}$. Moreover, this can be upgraded to a $\ell^2(L^p)$ decoupling if $D^2\phi$ is positive (or negative) semidefininte.
\end{thm}

\subsection{Decoupling for polynomials}
In this subsection, we introduce the main decoupling constants that will be studied in this paper.

Fix $n,d\ge 1$. We denote by $\mathcal P_{n,d}$ the collection of all $n$-variate polynomials $\phi$ of degree at most $d$, such that $\max_{[-1,1]^n}|\phi|\le 1$. Equivalently, the coefficients of $\phi$ are $\lesssim_{n,d} 1$.

For a parallelogram $R\sub \R^n$, its smallest dimension is referred to as {\it width}. See Section \ref{sec:notation} for the precise definitions.

\begin{defn}\label{defn:graphical_polynomial_decoupling}
    Let $n,d\ge 1$. We say that uniform decoupling holds for $\mathcal P_{n,d}$, or $U(n)$ holds, if for every $\phi\in \mathcal P_{n,d}$, $\delta\ll_{n,d} 1$ and $\eps\in (0,1]$, there exists a covering $\mathcal R_\delta=\mathcal R_\delta(\phi,\eps)$ of $[-1,1]^{n}$ by 
    parallelograms $R_\delta\sub [-2,2]^n$, such that the following holds:
    \begin{enumerate}
        \item \label{item:U_01} Each $R\in\mathcal R_\delta$ has width at least $\delta$, and $\mathcal R_\delta$ has bounded overlap. In particular, $\#\mathcal R_\delta\lesssim \delta^{-n}$.
        \item Every $R_\delta\in \mathcal R_\delta$ in $(\phi,\delta)$-flat.
        \item For every $2\le p\le \frac{2(n+2)}{n}$, $[-1,1]^{n}$ can be $\phi$-$\ell^p(L^p)$ decoupled into $\mathcal R_\delta$ at scale $\delta$ at cost $\lesssim \delta^{-\eps}$ as in Definition \ref{defn:graphical_decoupling}.
    \end{enumerate}    
    All implicit constants and the overlap function in this definition are allowed to depend on $n,d,\eps,p$,  but not on $\delta$ or the coefficients of $\phi$.
\end{defn}

\begin{defn}\label{defn:graphical_polynomial_decoupling_zero}
    Let $n,d\ge 1$. We say that (uniform) degenerate decoupling holds for $\mathcal P_{n,d}$, or $Z(n)$ holds, if the same statement of Definition \ref{defn:graphical_polynomial_decoupling} holds, except that we additionally have $\det D^2 \phi\equiv 0$.
\end{defn}

\begin{defn}\label{defn:sublevel_set_decoupling}
    Let $n,d\ge 1$. We say that (uniform) sublevel set decoupling holds for $\mathcal P_{n,d}$, or $S(n)$ holds, if for every $\phi\in \mathcal P_{n,d}$, $\delta\ll_{n,d} 1$ and $\eps\in (0,1]$, there exists a covering $\mathcal S_\delta$ of the sublevel set $\{x\in [-1,1]^n:|\phi(x)|<\delta\}$ by parallelograms $S_\delta\sub [-2,2]^n$, such that the following holds:
    \begin{enumerate}
        \item Each $S_\delta\in\mathcal S_\delta$ width at least $\delta$, and $\mathcal S_\delta$ has bounded overlap. In particular, $\#\mathcal S_\delta\lesssim \delta^{-n}$.
        \item $|\phi(x)|\lesssim \delta$ for every $x\in S_\delta\in \mathcal S_\delta$.
        \item For every $2\le p\le \frac{2(n+1)}{n-1}$, the sublevel set $\{x\in [-1,1]^n:|\phi(x)|\le\delta\}$ can be $\ell^p(L^p)$ decoupled into $\mathcal S_\delta$ at cost $\lesssim \delta^{-\eps}$ as in Definition \ref{defn:decoupling}.
    \end{enumerate}
    All implicit constants in this definition are allowed to depend on $n,d,\eps,p$, but not on $\delta$ or the coefficients of $\phi$.
\end{defn}
From the definitions, we see that $U(1)$ holds by \cite{Yang2} (where we can even show that $R_\delta$ has length at least $\delta^{1/2}$), and $U(2)$ holds by \cite[Theorem 5.22]{LiYang2024}. Hence $Z(1)$ and $Z(2)$ hold. Also, $S(1)$ follows from the fundamental theorem of algebra (together with a simple trick called $\delta$-thickening, see the proof of Proposition \ref{prop:delta_thickening} below), and $S(2)$ follows from \cite[Theorem 5.23]{LiYang2024}. The general case remains a conjecture.
\begin{conj}\label{conj:decoupling}
    The statements $U(n)$, $Z(n)$ and $S(n)$ hold for all $n\ge 1$.
\end{conj}
By the Taylor expansion in \cite[Section 2]{LiYang2023}, we have the following result.
\begin{prop}
    If $U(n)$ holds, then decoupling holds for any smooth functions $\phi$ of $n$ variables. More precisely, the same statement of Definition \ref{defn:graphical_polynomial_decoupling} holds for $\phi$, except that the implicit constants are allowed to depend on $\eps,p$ and the $C^{[1/\eps]}$ norm of $\phi$.
\end{prop}

\subsection{Main results}
We now proceed to state the main decoupling theorems in this paper. We start with two trivial observations.
\begin{lem}\label{lem:monotonicity}
    For $n\ge 1$, $U(n+1)\implies U(n)$, $Z(n+1)\implies Z(n)$ and $S(n+1)\implies S(n)$.
\end{lem}

\begin{lem}\label{lem:trivial_implications}
For $n\ge 1$, $S(n+1)\implies U(n)\implies Z(n)$.
\end{lem}
Note that $S(n+1)\implies U(n)$ follows from taking $\Phi(x,y)=y-\phi(x)$ for a given $\phi\in \mathcal P_{n,d}$.

Now we come to the nontrivial relations between these decoupling constants.
\begin{thm}\label{thm:main_U_imply_S}
    For $n\ge 1$, $U(n)$ implies $S(n+1)$.
\end{thm}
Indeed, in \cite{LiYang2023}, $S(2)$ (but without the width lower bound) is proved based on an iterative method using $U(1)$. This turns out to be true in all dimensions. 

Another main result proved in \cite{LiYang2023} is $U(2)$, which is proved by applying $Z(2)$ and establishing $S(2)$ (all without the width lower bound). Another main theorem of this paper is to prove that it is true in all dimensions.
\begin{thm}\label{thm:main_SZ_imply_U}
    For $n\ge 1$, $S(n)$ and $Z(n)$ together imply $U(n)$.
\end{thm}
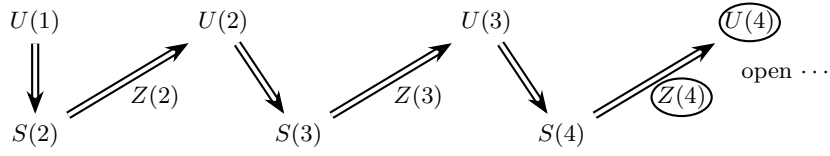
\begin{figure}
    \centering
    \begin{tikzpicture}[
    node distance=2.5cm and 2cm,
    every node/.style={font=\small},
    arrow/.style={-{Stealth}, double, double distance=1.5pt, thick}
  ]
  \node (u1) at (0,0) {$U(1)$};
  \node (s2) at (0,-1.5) {$S(2)$};
  \node (u2) at (2.5,0) {$U(2)$};
  \node (s3) at (3.5,-1.5) {$S(3)$};
  \node (u3) at (6,0) {$U(3)$};
  \node (s4) at (7,-1.5) {$S(4)$};
  \node (u4) at (9.5,0) {$U(4)$};
  
  \draw[arrow] (u1) -- (s2);
  \draw[arrow] (s2) -- (u2) coordinate[pos=0.25] (mid_s2_u2);
  \draw[arrow] (u2) -- (s3);
  \draw[arrow] (s3) -- (u3) coordinate[pos=0.25] (mid_s3_u3);
  \draw[arrow] (u3) -- (s4);
  \draw[arrow] (s4) -- (u4) coordinate[pos=0.25] (mid_s4_u4);
  
  \node[right=0.3cm of mid_s2_u2] (z2) {$Z(2)$};
  \node[right=0.3cm of mid_s3_u3] (z3) {$Z(3)$};
  \node[right=0.3cm of mid_s4_u4] (z4) {$Z(4)$};
  
  \draw[thick] (9.5,0) ellipse (0.4cm and 0.25cm);
  \draw[thick] (z4) ellipse (0.4cm and 0.25cm);
  
  \node at (10,-0.7) {open $\cdots$};
\end{tikzpicture}
    \caption{Implications between decoupling statements}
    \label{fig:implications}
\end{figure}

Combining Theorems \ref{thm:main_U_imply_S} and \ref{thm:main_SZ_imply_U}, we have the following corollary.
\begin{cor}\label{cor:conj_holds}
    For $n\ge 2$, $U(n-1)$ and $Z(n)$ together imply $U(n)$. As a result, if $Z(n)$ holds for all $n\ge 1$, then Conjecture \ref{conj:decoupling} holds.
\end{cor}

This means that in general, studying Conjecture \ref{conj:decoupling} is equivalent to studying decoupling for all polynomials $\phi\in \mathcal P_{n,d}$ satisfying $\det D^2\phi\equiv 0$. A natural approach is to classify all these polynomials, but it turns out to be rather difficult in higher dimensions, since it is connected to the Jacobi conjecture in algebraic geometry \cite{SingularHessian2004}. Nevertheless, in this paper, using the partial results established in \cite{SingularHessian2004}, we establish the following theorem.
\begin{thm}\label{thm:trivariate_decoupling}
The statement $Z(3)$ holds. In fact, the range of $p$ can be enlarged to $2\le p\le 4$.
\end{thm}
As a corollary, we obtain a uniform decoupling result for all trivariate polynomials, and a sublevel set decoupling for all $4$-variate polynomials.
\begin{cor}\label{cor:trivariate_decoupling}
    The statements $U(3)$ and $S(4)$ hold.
\end{cor}
The proof of $Z(3)$ uses the following key fact from \cite{SingularHessian2004}:
\begin{lem}\label{lem:zero_Hessian_paper}
    If $\phi(x_1,x_2,x_3)$ is a trivariate polynomial without affine terms that satisfies $\det D^2\phi\equiv 0$, then either of the following statements is correct:
    \begin{enumerate}
        \item There exists an affine bijection $\Xi$ on $\R^3$ such that $\phi\circ \Xi$ depends on $x_1,x_2$ only.
        \item There exists an affine bijection $\Xi$ on $\R^3$ such that $\phi\circ \Xi$ is of the form $A_1(x_1)+A_2(x_1)x_2+A_3(x_1)x_3$, where $A_1,A_2,A_3$ are univariate polynomials.
    \end{enumerate}
\end{lem}
The complete classification of polynomials $\phi\in \mathcal P_{n,d}$ with $\det D^2\phi\equiv 0$ is open for $n\ge 4$ (see \cite{SingularHessian2004,SingularHessian2018}). As a result, $Z(4)$ is open to us.

See Figure \ref{fig:implications}. By the main results of this paper, every statement to the left of (and including) $S(4)$ is known. Compared with \cite{LiYang2023}, the new results are $S(3),Z(3),U(3),S(4)$.

\subsection{Decoupling for homogeneous polynomials}
We can get better results if we restrict ourselves to the family of homogeneous polynomials.
\begin{defn}
    Let $n,d\ge 1$. We say that (uniform) homogeneous decoupling holds for $\mathcal P_{n,d}$, or $H(n)$ holds, if the same statement of Definition \ref{defn:graphical_polynomial_decoupling} holds, except that we additionally assume $\phi$ is homogeneous.
\end{defn}
Morally, we will show that $U(n)$ implies $H(n+1)$, namely, uniform decoupling for $n$-variate polynomials implies uniform decoupling for homogeneous $(n+1)$-variate polynomials. However, due to technical reasons, we can only prove a slightly weaker version. We first state an analogous assumption to $Z(n)$.
\begin{defn}\label{defn:IZ(n)}
    Fix $\eta:[-1,1]\to \R$ be $C^2$ and such that $\eta''\ne 0$. We say that $IZ(n)$ holds, if for every $\phi\in \mathcal P_{n,d}$ with $\det D^2 \phi\equiv 0$, $\delta\ll_{n,d} 1$ and $\eps\in (0,1]$, there exists a covering $\mathcal R_\delta$ of $[-1,1]^{n}$ by 
    parallelograms $R_\delta\sub [-2,2]^n$, such that the following holds:
    \begin{enumerate}
        \item Each $R\in\mathcal R_\delta$ has width at least $\delta$, and $\mathcal R_\delta$ has bounded overlap. In particular, $\#\mathcal R_\delta\lesssim \delta^{-n}$.
        \item Every $R_\delta\in \mathcal R_\delta$ in $(\phi,\delta)$-flat.
        \item Denote $\psi(x,y)=\eta(x)+\phi(y)$. For every $2\le p\le \frac{2(n+2)}{n}$, $[-1,1]^{n+1}$ can be $\psi$-$\ell^p(L^p)$ decoupled into 
        \begin{equation*}
            \{I\times R:I\in \mathcal I_\delta,R\in \mathcal R_\delta\}
        \end{equation*}
        at scale $\delta$ at cost $\lesssim \delta^{-\eps}$ as in Definition \ref{defn:graphical_decoupling}, where $\mathcal I_\delta$ denotes the tiling of $[-1,1]$ by intervals of length $\delta^{1/2}$.
    \end{enumerate}    
    All implicit constants and the overlap function in this definition are allowed to depend on $n,d,\eps,p,\norm{\eta}_{C^2},\inf |\eta''| $, but not on $\delta$ or the coefficients of $\phi$. 
\end{defn}
In general, proving $IZ(n)$ would require us to apply analogous properties to Lemma \ref{lem:zero_Hessian_paper}. In particular, we are able to prove the following proposition.
\begin{prop}\label{prop:IZ(3)}
    The statement $IZ(3)$ holds.
\end{prop}
The main theorem in this paper concerning decouping for homogeneous polynomials is as follows.
\begin{thm}\label{thm:decoupling_homo}
    For $n\ge 1$, $U(n)$ and $IZ(n)$ together imply $H(n+1)$. 
\end{thm}
We will outline the proof of both Proposition \ref{prop:IZ(3)} and Theorem \ref{thm:decoupling_homo} in Section \ref{sec:homo_highdim}, using a similar argument to \cite[Section 12]{LiYang2025}. Indeed, in \cite{LiYang2025}, we proved $H(3)$ based on the two principles introduced in \cite{LiYang2023}, a general sublevel set decoupling which requires $U(1)$, and a uniform bivariate decoupling for ``pseudo-polynomials" which requires $U(2)$. Also, $IZ(2)$ follows from \cite[Proposition 13.1]{LiYang2025}.

\begin{cor}\label{cor:H(4)}
    The statement $H(4)$ holds.
\end{cor}
\begin{proof}
    It follows from Proposition \ref{prop:IZ(3)}, Corollary \ref{cor:trivariate_decoupling} and Theorem \ref{thm:decoupling_homo}.
\end{proof}

\subsection{Outline of the paper} 
In Section \ref{sec:sub_parallelogram} we deal with a technical issue of decoupling a sub-parallelogram. In Section \ref{sec:degeneracy_locating}  we invoke a special case of the degeneracy locating principle established in \cite{LiYang2024}. In Section \ref{sec:S_imply_U} we prove Theorem \ref{thm:main_SZ_imply_U}. In Section \ref{sec:almost_affine}, we prove Theorem \ref{thm:trivariate_decoupling}. In Section \ref{sec:U_imply_S}, we prove Theorem \ref{thm:main_U_imply_S}. In Section \ref{sec:homo_highdim}, we prove Proposition \ref{prop:IZ(3)} and Theorem \ref{thm:decoupling_homo}.

\subsection{Notation}\label{sec:notation}

We introduce some notation that will be used in this paper. 

\begin{enumerate}
    \item \label{item:big_O} We use the standard notation $a=O_M(b)$, or $|a|\lesssim_M b$ to mean that there is a constant $C$ depending on some parameter $M$, such that $|a|\leq Cb$. When the dependence on the parameter $M$ is unimportant, we may simply write $a=O(b)$ or $|a|\lesssim b$. Similarly, we define $\gtrsim$ and $\sim$.


\item Given a subset $M\sub \R^n$, we define the $\delta$-neighbourhood $N_\delta(M)\sub \R^n$ to be
    \begin{equation*}
N_\delta(M) := \{ x + c: x \in M, c\in \R^n,|c|<\delta\}.
\end{equation*}

    \item Given a continuous function $\phi:\Omega\sub \R^n\to \R$, we define the $\delta$-vertical neighbourhood $N^\phi_\delta(\Omega)\sub \R^{n+1}$ by    
    \begin{equation*}
N_\delta^\phi(\Omega): =\{ (x , \phi(x) + c): x \in \Omega, c \in [-\delta,\delta]\}.
\end{equation*}

\item We denote by $\mathcal P_{n,d}$ the collection of all $n$-variate real polynomials $\phi$ of degree at most $d$, such that $\max_{[-1,1]^n}|\phi(x)|\le 1$.

\end{enumerate}

\subsubsection{Notation on parallelograms}\label{sec:equivalence_objects}
Decoupling inequalities are formulated using parallelograms in $\R^n$. To deal with technicalities, it is crucial that we make clear some fundamental geometric concepts.

\begin{enumerate}
    \item A parallelogram $R\sub \R^n$ is defined to be a set of the form
    \begin{equation*}
        \{x+b\in \R^n:|x\cdot u_i|\le l_i\},
    \end{equation*}
    where $\{u_i:1\le i\le n\}$ is a basis of $\R^n$ consisting of unit vectors, $b\in \R^n$ and $l_i\ge 0$. If $\{u_i:1\le i\le n\}$ is in addition orthogonal, we say that $R$ is a rectangle. Unless otherwise specified, we always assume a parallelogram in $\R^n$ has positive $n$-volume, namely, $l_i>0$ for all $1\le i\le n$. We say that two parallelograms are disjoint if their interiors are disjoint. The {\it width} of a parallelogram is defined by its shortest dimension, namely, $\min_i l_i$.

    \item Let $R\sub \R^n$ be a parallelogram. We denote by $\lambda_R$ an affine bijection from $[-1,1]^n$ to $R$. (There is more than one choice of $\lambda_R$, but in our paper, any such a choice will work identically, since in decoupling inequalities, the orientations of such affine bijections do not make any difference.) We say that an affine bijection $\lambda$ on $\R^n$ is {\it bounded} by $C$ if $\lambda x=Ax+b$ where all entries of $A,b$ are bounded by $C$. 

    \item Unless otherwise specified, for a parallelogram $S\sub \R^n$ and $C>0$, we denote by $CS$ the concentric dilation of $S$ by a factor of $C$. Given $t\in \R^n$, we denote $S+t=\{s+t:s\in S\}$. Note that in this notation, we have $C(S+t)=CS+t$.

\item Given a subset $S\sub \R^n$, a parallelogram $R\sub \R^n$ and $C\ge 1$, we say that $S,R$ are $C$-equivalent if $C^{-1}R\sub S\sub CR$.  If the constant $C$ is unimportant, we simply say that $S$ is equivalent to $R$, or that $S$ is an {\it almost parallelogram}.

    \item For $0<\delta<1$, by a {\it tiling} of a parallelogram $R\sub \R^n$ by cubes of side length $\delta$, we mean a covering $\mathcal T$ of $R$ by translated copies $T$ of a cube of side length $\delta$, such that different $T$ and $T'$ from $\mathcal T$ have disjoint interiors, and that $T\sub 2R$.

    \item By the John ellipsoid theorem \cite{John}, every convex body in $\R^n$ is $C_n$-equivalent to a rectangle in $\R^n$, for some dimensional constant $C_n$. For this reason and in view of the enlarged overlap introduced right below, we often do not distinguish rectangles from parallelograms.

\end{enumerate}

\subsubsection{Enlarged overlap}

\begin{defn}\label{defn:enlarged_overlap}
By an {\it overlap function} we mean an increasing function $B:[1,\infty)\to [1,\infty)$. Given a finite family $\mathcal R$ of parallelograms in $\R^n$ and an overlap function $B$, we say $\mathcal R$ is $B$-overlapping, or $B$ is an overlap function of $\mathcal R$, if for every $\mu\ge 1$ we have
    \begin{equation*}
        \sum_{R\in \mathcal R}1_{\mu R}\le B(\mu).
    \end{equation*}
\end{defn}
Given $\delta>0$ and a family $\mathcal R$ of parallelograms depending on $\delta$ (and other factors such as $\eps>0$). We say that $\mathcal R$ is boundedly overlapping, if its overlap function $B$ does not depend on $\delta$ (but it might depend on $\eps$ and other factors).

\subsection{Acknowledgement}
The first author was supported by AMS-Simons Travel Grants. The second author was supported by the Croucher Fellowships for Postdoctoral Research. The authors thank Terence Tao for insightful discussions.

\section{Decoupling for a sub-parallelogram}\label{sec:sub_parallelogram}

For $\ell^2$ decoupling, the decoupling of $[-1,1]^n$ automatically implies the decoupling of a parallelogram $R\sub [-1,1]^n$, as there is no cardinality term on the right hand side of \eqref{eqn:defn_decoupling}. However, for $\ell^p$ decoupling when $p>2$, if we use the decoupling of $[-1,1]^n$ to trivially obtain the decoupling of $R$, then the cardinality term will be highly inefficient, since many rectangles from the decoupling of $[-1,1]^n$ do not intersect $R$. 

In our proofs, it is common for us to decouple a parallelogram $R\sub [-1,1]^n$; since we are in the case of families of polynomials which are rescaling invariant, it seems that we can always rescale $R$ to $[-1,1]^n$ and decouple the rescaled polynomial instead. However, since we are interested in width lower bounds of the decoupled rectangles, this issue becomes more subtle.

We first start with $\phi$-decoupling for polynomials $\phi\in \mathcal P_{n,d}$, by introducing a superficially stronger version $\tilde U(n)$ of $U(n)$ as in Definition \ref{defn:graphical_polynomial_decoupling}. Indeed, after a few propositions and theorems, we shall later see in Corollary \ref{cor:U_tilde_equivalent} that $U(n)$ is equivalent to $\tilde U(n)$.
\begin{defn}
Let $n,d\ge 1$. We say that $\tilde U(n)$ holds, if for every $\phi\in \mathcal P_{n,d}$, every $\delta\ll_{n,d} 1$, every $\eps\in (0,1]$ and every parallelogram $R\sub [-1,1]^n$ with width $w\ge \delta$, there exists a covering $\mathcal R_\delta$ of $R$ by 
    parallelograms $R_\delta\sub 2R$, such that the following holds:
    \begin{enumerate}
        \item Each $R\in\mathcal R_\delta$ has width at least $\delta$, and $\mathcal R_\delta$ has bounded overlap. In particular, $\#\mathcal R_\delta\lesssim \delta^{-n}$.
        \item Every $R_\delta\in \mathcal R_\delta$ in $(\phi,\delta)$-flat.
        \item For every $2\le p\le \frac{2(n+2)}{n}$, $R$ can be $\phi$-$\ell^p(L^p)$ decoupled into $\mathcal R_\delta$ at scale $\delta$ at cost $\lesssim \delta^{-\eps}$ as in Definition \ref{defn:graphical_decoupling}.
    \end{enumerate}
    All implicit constants and the overlap function in this definition are only allowed to depend on $n,d,\eps,p$.
\end{defn}
It is trivial to see that $\tilde U(n)\implies U(n)$. The more interesting part is the following partial converse.
\begin{prop}\label{prop:tilde_U(n)}
    For $n\ge 1$, $U(n)$ and $\tilde U(n-1)$ together imply $\tilde U(n)$.
\end{prop}
Note that \cite{Yang2} proves that $\tilde U(1)$ holds, since it is a $\ell^2$ decoupling. Therefore, by Proposition \ref{prop:tilde_U(n)}, proving that $\tilde U(n)$ holds for all $n$ is equivalent to proving that $U(n)$ holds for all $n$. We also have the following corollary.
\begin{cor}\label{cor:U(2)}
    The statement $\tilde U(2)$ holds.
\end{cor}
\begin{proof}
    By \cite[Theorem 5.22]{LiYang2024}, $U(2)$ holds. By \cite[Theorem 1.4]{Yang2}, $\tilde U(1)$ holds. Then the result follows from Proposition \ref{prop:tilde_U(n)}.
\end{proof}
\fbox{Notation} All implicit constants below in this section are allowed to depend on $n,d$.

\begin{proof}[Proof of Proposition \ref{prop:tilde_U(n)}]
    Given $R\sub [-1,1]^n$ with width $w$, we first apply an intermediate decoupling at scale $w$ (every decoupling in this proof is $\ell^p(L^p)$ where $2\le p\le \frac{2(n+2)}{n}$). But at this scale, $\phi$ is indistinguishable with $\phi|_T$, where $T$ is the lower dimensional parallelogram obtained from $R$ by removing one shortest dimension of $R$. Applying $\tilde U(n-1)$, we can $\phi|_T$-decouple $T$ at scale $w$. This equivalently $\phi$-decouples $R$ into $(\phi,w)$-flat parallelograms $R'$ each having width at least $w$.

    We then rescale $R'$ to $[-1,1]^n$, and note that 
    \begin{equation*}
        \phi\circ \lambda_{R'}=Cw \psi+\text{affine terms},
    \end{equation*}
    for some $C\lesssim 1$ and $\psi\in \mathcal P_{n,d}$. Applying $U(n)$ at scale $w^{-1}\delta$, we may $\psi$-decouple $[-1,1]^n$ into $(\psi,w^{-1}\delta)$-flat parallelograms $R''\sub [-2,2]^n$ of width at least $w^{-1}\delta$. Rescaling back, we have $\phi$-decoupled $R'$ into $(\phi,O(\delta))$-flat parallelograms $\lambda_{R'}(R'')\sub 2R'\sub 2R$ with width at least $ww^{-1}\delta=\delta $. We may further apply a trivial partition of $\lambda_{R'}(R'')$ into $O(1)$ parallelograms of width still at least $\delta$, such that they are $(\phi,\delta)$-flat (see \cite[Proposition 5.24]{LiYang2024} for details).
\end{proof}

We next consider an analogue of Proposition \ref{prop:tilde_U(n)} for the sublevel set decoupling in Definition \ref{defn:sublevel_set_decoupling}. Similarly to the discussion of $\tilde U(n)$, in Corollary \ref{cor:U_tilde_equivalent}, we see that $S(n)$ is equivalent to $\tilde S(n)$.

\begin{defn}\label{defn:sublevel_set_decoupling_tilde}
    Let $n,d\ge 1$. We say that $\tilde S(n)$ holds, if for every $\phi\in \mathcal P_{n,d}$, every $\delta\ll_{n,d} 1$, every $\eps\in (0,1]$ and every parallelogram $R\sub [-1,1]^n$ with width at least $\delta$, there exists a covering $\mathcal S_\delta$ of the sublevel set $\{x\in R:|\phi(x)|\le \delta\}$ by parallelograms $S_\delta\sub 2R$, such that the following holds:
    \begin{enumerate}
        \item Each $S_\delta\in\mathcal S_\delta$ has width at least $\delta$, and $\mathcal S_\delta$ has bounded overlap. In particular, $\#\mathcal S_\delta\lesssim_\eps \delta^{-n}$.
        \item $|\phi(x)|\lesssim \delta$ for every $x\in S_\delta\in \mathcal S_\delta$.
        \item For every $2\le p\le \frac{2(n+1)}{n-1}$, the sublevel set $\{x\in R:|\phi(x)|\le\delta\}$ can be $\ell^p(L^p)$ decoupled into $\mathcal S_\delta$ as in Definition \ref{defn:decoupling} at cost $O_\eps (\delta^{-\eps})$.
    \end{enumerate}
    All implicit constants and the overlap function in this definition are only allowed to depend on $n,d,\eps,p$.
\end{defn}

We first settle the case $\tilde S(1)$. We present the proof here since it uses a little trick we call $\delta$-thickening, which will be used several times in this paper.
\begin{prop}\label{prop:delta_thickening}
    The statement $\tilde S(1)$ holds.
\end{prop}
\begin{proof}[Proof of Proposition \ref{prop:delta_thickening}]
By the fundamental theorem of algebra, for $R\sub [-1,1]$, the sublevel set $\{x\in R:|\phi(x)|<\delta\}$ is a union of $O_d(1)$ many intervals $I$. A priori, the lengths of $I$ may be small, but we may simply replace $I$ by $N_\delta(I)$, on each of which we have $|\phi|\lesssim \delta$.
\end{proof}
Obviously, $\tilde S(n)$ implies $S(n)$. We also have the following partial converse. 
\begin{prop}\label{prop:tilde_S(n)}
    For $n\ge 1$, $S(n)$ and $\tilde S(n-1)$ together imply $\tilde S(n)$.
\end{prop}
Therefore, by Proposition \ref{prop:tilde_S(n)}, proving that $\tilde S(n)$ holds for all $n$ is equivalent to proving that $S(n)$ holds for all $n$.
\begin{proof}[Proof of Proposition \ref{prop:tilde_S(n)}]
    Given $R\sub [-1,1]^n$ with width $w$, we first apply an intermediate decoupling at scale $w$ (every decoupling in this proof is $\ell^p(L^p)$ where $2\le p\le \frac{2(n+1)}{n-1}$). But at this scale, $\phi$ is indistinguishable with $\phi|_T$, where $T$ is the lower dimensional parallelogram obtained from $R$ by removing one shortest dimension of $R$. Applying $\tilde S(n-1)$, we can decouple $T$ at scale $w$. This equivalently decouples $R$ into parallelograms $R'$ each having width equal to $w$, on each of which we have $|\phi|\lesssim w$.

    It remains to decouple $\{x\in R':|\phi(x)|<\delta\}$ into parallelograms on which $|\phi|\lesssim \delta$.

    We then rescale $R'$ to $[-1,1]^n$, and note that 
    \begin{equation*}
        \phi\circ \lambda_{R'}=Cw \psi,
    \end{equation*}
    for some $C\lesssim 1$ and $\psi\in \mathcal P_{n,d}$. Applying $S(n)$ at scale $w^{-1}\delta$, we may decouple $[-1,1]^n$ into parallelograms $R''\sub [-2,2]^n$ of width at least $w^{-1}\delta$, on each of which we have $|\psi|\lesssim w^{-1}\delta$. Rescaling back, we have decoupled $R'$ into parallelograms $\lambda_{R'}(R'')\sub 2R'\sub 2R$ with width at least $ww^{-1}\delta=\delta $, on each of which we have $|\phi|\lesssim \delta$.
\end{proof}

\begin{cor}\label{cor:S(2)}
    The statement $\tilde S(2)$ holds.
\end{cor}
\begin{proof}
    By \cite[Theorem 5.23]{LiYang2024}, $S(2)$ holds. By Proposition \ref{prop:delta_thickening}, $\tilde S(1)$ holds. Then the result follows from Proposition \ref{prop:tilde_S(n)}.
\end{proof}

We may also slightly upgrade Lemma \ref{lem:trivial_implications}.
\begin{lem}\label{lem:trivial_implications_tilde}
    For $n\ge 1$, $S(n+1)\implies \tilde U(n)$.
\end{lem}
\begin{proof}
We prove by induction on $n$. For $n=1$, $\tilde U(1)$ holds by \cite{Yang2}. Assuming that $S(n)\implies \tilde U(n-1)$ for some $n\ge 2$, we will prove that $S(n+1)\implies \tilde U(n)$.

By Lemma \ref{lem:trivial_implications}, we have $S(n+1)\implies U(n)$. But by Lemma \ref{lem:monotonicity}, we have $S(n+1)\implies S(n)$, which implies $\tilde U(n-1)$ by the induction hypothesis. Then the result follows from Proposition \ref{prop:tilde_U(n)}.
\end{proof}

\section{The degeneracy locating principle}\label{sec:degeneracy_locating}
In this section, we invoke a special case of the degeneracy locating principle of decoupling introduced in \cite[Section 2]{LiYang2024}. We only introduce and state the terminology and notation necessary for this paper.

\fbox{Notation} All implicit constants are allowed to depend on $n,d$ in this section.

\subsection{Setup}

Denote by $\mathcal R_0$ the collection of all parallelograms contained in $[-2,2]^n$. Denote by $\mathcal A_0$ the collection of all bounded affine bijections on $\R^n$ (see Section \ref{sec:notation}). The following definition corresponds to Definitions 3.6 and 3.7 of \cite{LiYang2024}.

\begin{defn}[Rescaling invariant pair]\label{defn:rescaling_invariant_pair}
Let $\mathcal P\sub \mathcal P_{n,d}$, $\mathcal R\sub \mathcal R_0$ and $\mathcal A\sub \mathcal A_0$ be subfamilies. We say that the pair of families $(\mathcal P,\mathcal R)$ is $\mathcal A$-rescaling invariant if the following holds.
\begin{enumerate}
    \item $\mathcal A$ is closed under composition, namely, $\Xi_1\circ \Xi_2\in \mathcal A$ whenever $\Xi_1,\Xi_2\in \mathcal A$.
    \item $\Xi(R)\in \mathcal R$ whenever $\Xi\in \mathcal A$ and $R\in \mathcal R$.
    \item There exists a constant $C=C(n,d)\ge 1$ such that the following holds. For every $\sigma\in (0,1]$, $\phi \in \mathcal P$ and $R\in \mathcal R$ that is $(\phi,\sigma)$-flat, there exists some $\psi\in \mathcal P$ such that
    \begin{equation*}
        \phi\circ \lambda_R=C\sigma \psi+\text{affine terms}.
    \end{equation*}
\end{enumerate}
\end{defn}
We have the following easy observation (see for instance \cite[Section 7]{LiYang2023}). 
\begin{prop}\label{prop:polynomial_rescaling_invariant}
    $(\mathcal P_{n,d},\mathcal R_0)$ is $\mathcal A_0$-rescaling invariant.
\end{prop}

\begin{defn}[Degeneracy determinant]\label{defn:degeneracy_determinant}
    Let $\mathcal P\subseteq \mathcal P_{n,d}$ and $\mathcal R \subseteq \mathcal{R}_0$. We say that an operator $H$ that maps $\phi \in \mathcal P$ to some $n$-variate polynomial is a degeneracy determinant, if there exist constants $C,C',C''\ge 1$ and $\beta\in (0,1]$, depending only on $\mathcal P,\mathcal R$, such that the following holds:

    \begin{enumerate}
        \item \label{item:Lipschitz} (Lipschitz) For every $\phi\in \mathcal P$, 
        \begin{equation}\label{eqn:Lipschitz}
            |H\phi(x)-H\phi(y)|\le C|x-y|,\quad \forall x,y\in \R^n.
        \end{equation}
        
        \item \label{item:totally_degenerate_approximation} (Totally degenerate approximation)
        Let $\phi \in \mathcal P$ be such that $|H \phi(x)|\le \sigma$ on $[-1,1]^n$ for some $\sigma\in (0,1]$. Then there exists some $\psi\in \mathcal P$ satisfying $H\psi\equiv 0$ on $\R^n$, and that            \begin{equation}\label{eqn:small_hessian_error}
            \|\phi-C'\psi\|_{L^\infty([-1,1]^n)}\leq C\sigma^{\beta}.
        \end{equation}

        \item \label{item:regularity_small_Hessian_rescaling} (Regularity under rescaling) Let $R\in \mathcal R$. Then we have
        \begin{equation}
            \mu^{C''}|H\phi(\lambda_{R}x)|\le |H(\phi\circ \lambda_{R})(x)|\le |H\phi(\lambda_{R}x)|,\quad \forall x\in \R^n,
        \end{equation}
    where $\mu$ is the width of $R$.
        \end{enumerate} 

\end{defn}
{\it Remark.} This is a combination of Definitions 3.13 and 3.14 of \cite{LiYang2024}, where we have omitted the terminology ``$\mathfrak R$-regular" there.

We invoke \cite[Corollary 3.16]{LiYang2024}.
\begin{cor}\label{cor:hessian}
    Let $\mathcal P=\mathcal P_{n,d}$ and let $\mathcal R=\mathcal R_0$. Then $H\phi:=\det D^2\phi$ is a degeneracy determinant on $\mathcal P$.
\end{cor}

\begin{defn}[Trivial covering property]\label{defn:trivial_covering}
    Let $(\mathcal P,\mathcal R)$ be $\mathcal A$-rescaling invariant. 
    We say that $(\mathcal P,\mathcal R)$ satisfies the trivial covering property if there exist constants $C_1,C_2,C\ge 1$, depending only on $\mathcal P,\mathcal R,\mathcal A$, such that the following holds for every $\phi\in \mathcal P$: there exists a family $\mathcal R_{\mathrm{tri}}\sub \mathcal R$ such that for every $K\ge 1$, we can cover $[-1,1]^n$ by $\le CK^{C_1}$ parallelograms $R_0\in \mathcal R_{\mathrm{tri}}$, such that each $R_0$ has length $\le K^{-1}$ and width $\ge C^{-1}K^{-C_2}$. Moreover, the affine bijection $\lambda_{R_0}\in \mathcal A$.
        
\end{defn}
For example, in the case where $\mathcal A=\mathcal A_0$ and $\mathcal R$ contains of all the (axis-parallel) squares contained in $[-1,1]^n$, we may simply cover $[-1,1]^k$ by squares of side length $\sim K^{-1}$. In this case, we may take $C_1=k$, $C_2=1$.

\subsection{Statement of theorem}

We are now ready to state the the degeneracy locating principle adapted and simplified to the setting of this paper. In the following, all decoupling inequalities are assumed to be $\ell^p(L^p)$ decoupling, where $2\le p\le \frac{2(n+2)}{n}$.

\begin{thm}[Degeneracy locating principle]\label{thm:degeneracy_locating_principle}
Let $(\mathcal P,\mathcal R)$ be a pair that is $\mathcal A$-rescaling invariant, and assume it satisfies the trivial covering property. Let $H$ be a degeneracy determinant.

For $\phi\in \mathcal P$, there exist subfamilies $\mathcal R_{\mathrm{sub}}$, $\mathcal R_{\mathrm{deg}}$, $\mathcal R_{\mathrm{non}}$ of $\mathcal R$, such that the following holds for every $\sigma>0$ and every $\eps>0$, where all constants $C$ (resp. $C_\eps$ and the overlap function $B$) below depend only on $\mathcal P,\mathcal R,\mathcal A,H$ (resp. $\mathcal P,\mathcal R,\mathcal A,H,\eps$):
    \begin{enumerate}[label=(\alph*)]
        \item \label{item:sublevel_set_decoupling} (Sublevel set decoupling) The subset
        \begin{equation*}
            \{x\in [-1,1]^n:|H\phi(x)|\le \sigma\}
        \end{equation*}
        can be decoupled (as in Definition \ref{defn:decoupling}) into $B$-overlapping parallelograms $R_{\mathrm{sub}}\in \mathcal R_{\mathrm{sub}}$ at a cost of $C_\varepsilon\sigma^{-\varepsilon}$, on each of which $|H\phi|\le C\sigma$. Moreover, we require that $\lambda_{R_{\mathrm{sub}}}\in \mathcal A$, and that $R_{\mathrm{sub}}$ has width at least $\sigma$.
    
    \item \label{item:degnerate_decoupling} (Degenerate decoupling) If $H\phi\equiv 0$ on $[-1,1]^n$, then $[-1,1]^n$ can be $\phi$-decoupled (as in Definition \ref{defn:graphical_decoupling}) at scale $\sigma$ into $(\phi,\sigma)$-flat $B$-overlapping parallelograms $R_{\mathrm{deg}}\in \mathcal R_{\mathrm{deg}}$ at a cost of $C_\eps\sigma^{-\varepsilon}$. Moreover, we require that $\lambda_{R_{\mathrm{deg}}}\in \mathcal A$, and that $R_{\mathrm{deg}}$ has width at least $\sigma$. (Importantly, the power $C$ on $K$ cannot depend on $\eps$.)

    \item \label{item:nondegnerate_decoupling} (Nondegnerate decoupling) If $|H\phi| \geq K^{-1}$ on $[-1,1]^n$ for some $K\ge 1$, then $[-1,1]^n$ can be $\phi$-decoupled (as in Definition \ref{defn:graphical_decoupling}) at scale $\sigma$ into $(\phi,\sigma)$-flat $B$-overlapping parallelograms $R_{\mathrm{non}}\in \mathcal R_{\mathrm{non}}$ at a cost of $K^{C}\sigma^{-\varepsilon}$, and that $R_{\mathrm{non}}$ has width at least $\sigma$.

\item [(*)]\label{item:star} (Lower dimensional decoupling)
Given a parallelogram $R$ that is either an element of $\mathcal R_{\mathrm{tri}}$ from Definition \ref{defn:trivial_covering}, or an element of $\mathcal R_{\mathrm{sub}}$ in Assumption \ref{item:sublevel_set_decoupling}. Denote by $w$ the shortest dimension of $R$, and let $L\sub \R^{n-1}$ be the lower dimensional parallelogram obtained from $R$ by removing its shortest dimension. Then the following conditions hold:
            \begin{enumerate}[label=(\roman*)]
                \item \label{item:i} $L$ can be $\phi|_L$-decoupled into $(\phi|_L, w)$-flat parallelograms $R'_{\mathrm{low}}$ at scale $ w$ at cost $C_\eps  w^{-\eps}$, with $R'_{\mathrm{low}}$ having width at least $ w$.
                \item \label{item:ii} Denote by $R_{\mathrm{low}}$ the parallelogram obtained from $R'_{\mathrm{low}}$ by adding back the dimension of length $w$. Then we require that $\lambda_{R_{\mathrm{low}}}\in \mathcal A$, and that $R_{\mathrm{low}}$ is contained in $2R$.     
            \end{enumerate}

        \end{enumerate}
Then for every $\phi\in \mathcal P$, $\delta>0$ and $\eps>0$, $[-1,1]^n$ can be $\phi$-decoupled into $(\phi,\delta)$-flat $B'$-overlapping parallelograms $R_{\mathrm{final}}\in \mathcal R$ at scale $\delta$ at cost $C'_\eps \delta^{-\eps}$, where the width of $R_{\mathrm{final}}$ are at least $\delta$, and $C'_\eps$ and $B'$ depend only on $\mathcal P,\mathcal R,\mathcal A,H,\eps,p,q$.
        
\end{thm}

\section{Sublevel set and degenerate decoupling imply uniform decoupling}\label{sec:S_imply_U}

In this section, we prove Theorem \ref{thm:main_SZ_imply_U}. The main idea is to use the degeneracy locating principle Theorem \ref{thm:degeneracy_locating_principle}, in the following setting.

Let $\mathcal P=\mathcal P_{n,d}$, $\mathcal A=\mathcal A_0$ and $\mathcal R=\mathcal R_0$. Define $H\phi(x):=\det D^2 \phi$ for $\phi\in \mathcal P$.

We then need to verify the following ingredients of the degeneracy locating principle for $\mathcal P,\mathcal R,\mathcal A,H$:
\begin{itemize}
    \item $(\mathcal P,\mathcal R)$ is $\mathcal A$-rescaling invariant: this is Proposition \ref{prop:polynomial_rescaling_invariant}.
    \item $H$ is a degeneracy determinant: this follows from Corollary \ref{cor:hessian}.
    \item Trivial covering property: this is trivial, since every cube of side length $K^{-1}$ belongs to $\mathcal R$. 
    \item Sublevel set decoupling: this is just the assumption $S(n)$.
    \item Totally degenerate decoupling: this is just the assumption $Z(n)$.
    \item Nondegenerate decoupling: this follows from Theorem \ref{thm:Bourgain_Demeter}.
    \item Lower dimensional decoupling: we have $\tilde U(n-1)$, which follows from the assumption $S(n)$ and Lemma \ref{lem:trivial_implications_tilde}. Hence (*) follows. 
\end{itemize}

\section{Decoupling for nearly affine functions}\label{sec:almost_affine}

In this section we prove Theorem \ref{thm:trivariate_decoupling}. In fact, we will prove the following theorem in all dimensions $n$. 

\begin{thm}\label{thm:almost_affine}
    Let $n\ge 1$. Let $A_1,\dots,A_n\in \mathcal P_{1,d}$, and define
    \begin{equation}\label{eqn:phi_almost_affine}   \phi(x)=\phi(x_1,\dots,x_n)=A_1(x_1)+A_2(x_1)x_2+\cdots+A_n(x_1)x_n.
    \end{equation}    
    Then for $2\le p\le 4$, $\delta\in (0,1)$ and $\eps\in (0,1)$, $[-1,1]^n$ can be $\phi$-$\ell^p(L^p)$ decoupled into boundedly overlapping $(\phi,\delta)$-flat parallelograms at the cost of $C_\eps \delta^{-\eps}$, where $C_\eps$ and the overlap function do not depend on $\phi,\delta$. Moreover, the parallelograms have width at least $\delta$, and can be taken from the family
    \begin{equation}\label{eqn:I_times_T}
     \mathcal R^0:=\{I\times T:I\sub [-1,1]\text{ is an interval, $T\sub [-1,1]^{n-1}$ is a parallelogram}\}.
 \end{equation}
\end{thm}
\begin{proof}[Proof of Theorem \ref{thm:trivariate_decoupling} assuming Theorem \ref{thm:almost_affine}]
It suffices to prove decoupling for the two types of polynomials in Lemma \ref{lem:zero_Hessian_paper}. The decoupling for the first type of functions follows from \cite[Theorem 5.22]{LiYang2024}. The decoupling for the second type of functions is the special case of Theorem \ref{thm:almost_affine} when $n=3$.
\end{proof}

To prove Theorem \ref{thm:almost_affine}, we use induction on $n$ and the degeneracy locating principle Theorem \ref{thm:degeneracy_locating_principle}, this time with different families and a different degeneracy determinant.

The case $n=1$ follows from \cite{Yang2}. Assume that it holds for dimension $n-1\ge 1$, and we will prove that it holds for dimension $n$.

\subsection{Ingredients of degeneracy locating principle}\label{eqn:sec:delta_thickening}

Let $\mathcal P^0$ consist of the graphs over $[-1,1]^n$ of all polynomials $\phi$ in the form of \eqref{eqn:phi_almost_affine}. 

Let $\mathcal A^0$ consist of all bounded affine bijections the form $\Xi x:=Vx+w$ where $V$ is $n\times n$ and $V_{1j}=V_{j1}=0$ for all $2\le j\le n$. (Equivalently, $\mathcal A^0$ consists of all $\lambda_R$ where $R\in \mathcal R^0$ defined in \eqref{eqn:I_times_T}.) Define
for $\phi\in \mathcal P^0$
\begin{equation}\label{eqn:H^0}
    H^0 \phi(x):=\sum_{i=1}^n (A_i'(x_1))^2.
\end{equation}

We then need to verify the following ingredients of the degeneracy locating principle for $\mathcal P^0,\mathcal R^0,\mathcal A^0,H^0$:
\begin{itemize}
    \item $(\mathcal P^0,\mathcal R^0)$ is $\mathcal A^0$-rescaling invariant: this follows from the same proof of Proposition \ref{prop:polynomial_rescaling_invariant}.

    \item $H^0$ is a degeneracy determinant: this is straightforward to check, with $\beta=1/2$.
    \item Trivial covering property: this is trivial, since every cube $R_0$ of side length $K^{-1}$ belongs to $\mathcal R^0$. 
    \item Sublevel set decoupling: this follows from the same proof of Proposition \ref{prop:delta_thickening}.
    \item Totally degenerate decoupling: this follows from $\tilde U(1)$.
    \item Nondegenerate decoupling: this follows from Proposition \ref{prop:Pramanik_Seeger} below (with $K^{-1}$ replaced by $K^{-2}$).
    \item Lower dimensional decoupling: for parallelograms $R\in \mathcal R^0$ with width $w$, we need to consider two cases, namely, $w$ is attained by $I$ or $T$.
    \begin{itemize}
    \item If $T$ has width $w$, then the decoupling at scale $w$ follows from the induction hypothesis that Theorem \ref{thm:almost_affine} holds for $n-1$ variables.
    
        \item If $I$ has length $w$, then decoupling at scale $w$, all terms $A_i(x_1)$ become affine terms, so by affine invariance, we are decoupling the graph of
    \begin{equation*}
        (x_1,\dots,x_n)\mapsto x_1\Lambda(x_2,\dots,x_n),
    \end{equation*}
    for some linear transformation $\Lambda:\R^{n-1}\to \R$. By an affine transformation, this essentially becomes the decoupling of the function $\phi(x)=x_1x_2$. This follows from $\tilde U(2)$, which holds by Corollary \ref{cor:U(2)}.
    
    \end{itemize}
   
\end{itemize}

\begin{prop}[Nondegenerate decoupling]\label{prop:Pramanik_Seeger}
Let $n\ge 2$. Let $A_1,\dots,A_n\in \mathcal P_{1,d}$, and define
    \begin{equation}    \phi(x_1,\dots,x_n)=A_1(x_1)+A_2(x_1)x_2+\cdots+A_n(x_1)x_n.
    \end{equation}  
Assume also that $|H^0\phi(x_1)|\ge K^{-2}$ for all $x_1\in [-1,1]$ (see \eqref{eqn:H^0}). Then for $2\le p\le 4$, $\delta\in (0,1)$ and $\eps\in (0,1)$, $[-1,1]^n$ can be $\phi$-$\ell^p(L^p)$ decoupled into boundedly overlapping $(\phi,\delta)$-flat parallelograms $R$ at cost $C_\eps K^{C}\delta^{-\eps}$, where $C_\eps$ and the overlap function depend only on $n,d,\eps,p$ and $C$ depends only on $n,d$.

Moreover, $R$ can be taken to be of the following form:
\begin{equation}\label{eqn:decoupling_parallelograms}
    [u,u+\delta^{1/2}]\times \{(x_2,\dots,x_n)\in [-1,1]^{n-1}:A_2'(u)x_2+\cdots A_n'(u)x_n\in [v,v+\delta^{1/2}]\},
\end{equation}
where $u,v$ range through the $\delta^{1/2}$-lattice on $[-1,1]$ and $[-O(1),O(1)]$, respectively. In particular, $R$ has width $\gtrsim \delta^{1/2}\ge \delta$. 
\end{prop}

\fbox{Notation} All implicit constants below in this section are allowed to depend on $n,d$. All decoupling inequalities below are $\ell^p(L^p)$ where $2\le p\le 4$.

\subsection{Proof of Proposition \ref{prop:Pramanik_Seeger}}

    For $n=2$, the result just follows from Theorem \ref{thm:Bourgain_Demeter} since $|\det D^2 \phi|\gtrsim K^{-2}$.

    For $n> 2$, we will reduce to the case $n=2$. We can check that $\det D^2 \phi\equiv 0$, so we will use a Pramanik-Seeger iteration similar to the way that is used to prove decoupling for the light cone in \cite[Section 8]{BD2015} (see also \cite[Section 2]{LiYang2024}). 

    \subsubsection{Reduction to $|A_n'(x_1)|\sim 1$}

First, by Lipschitz condition of $H^0$, we may trivially partition $[-1,1]^n$ into $\sim K^n$ smaller cubes on each of which there is one $j$ such that $|A_j'(x_1)|\gtrsim K^{-1}$. By symmetry, it suffices to study the case $j=n$.
    
By a dyadic decomposition, we first trivially decouple into intervals $I_0$ where $|A_n'(x_1)|$ is dyadically a constant $\ge K^{-1}$. We then rescale $I_0$ to $[-1,1]$, so that $|\frac {d}{dx_1}A_n(\lambda_{I_0}(x_1))|$ is dyadically a constant $\sigma$. We may assume without loss of generality that each $I_0$ has length $\ge K^{-1}$, so that $\sigma\ge K^{-2}\gg \delta$. Denote
\begin{equation*}
\bar A_i(x_1):=A_i(\lambda_{I_0} x_1)-\text{constant term},\quad 1\le i\le n-1,
\end{equation*}
and 
\begin{equation*}
    \bar A_n(x_1)=\sigma^{-1}A_n(\lambda_{I_0} x_1)-\text{constant term},
\end{equation*}
so that $\bar A_i\in \mathcal P_{1,d}$ and $|\bar A_n'|\sim 1$. We will $\bar \phi$-decouple $[-1,1]^n$ where
\begin{equation*}
    \bar \phi(x):=\bar A_1(x_1)+\bar A_2(x_1)x_2+\cdots+\sigma\bar A_{n}(x_1)x_n.
\end{equation*}
Next, we apply a trivial decoupling to the coordinates $x_2,\dots,x_{n-1}$ to obtain intervals of length $\sigma$. Redefining $\bar A_1$ if necessary, it is equivalent to $\hat \phi$-decoupling $[-1,1]^n$ where
\begin{equation*}
    \hat \phi(x):=\bar A_1(x_1)+ \sigma \bar A_2(x_1)x_2+\cdots+\sigma\bar A_{n}(x_1)x_n.
\end{equation*}
Then, we apply a trivial decoupling to the coordinate $x_1$ to obtain intervals $I_1$ of length $\sigma$. Denote
\begin{equation*}
    \tilde A_1(x_1):=\sigma^{-2}\bar A_1(\lambda_{I_1} x_1)-\text{affine term},
\end{equation*}
and
\begin{equation*}
\tilde A_i(x_1):=\sigma^{-1}\bar A_i(\lambda_{I_1} x_1)-\text{constant term},\quad 2\le i\le n,
\end{equation*}
so that 
\begin{equation*}
    \tilde \phi(x):=\sigma^{-2}\hat \phi(\lambda_{I_1}x_1,x_2,\dots,x_n)-\text{affine terms}
\end{equation*}
satisfies that 
\begin{equation*}
    \tilde \phi(x)=\tilde A_1(x_1)+\tilde A_2(x_1)x_2+\cdots+\tilde A_n(x_1)x_n,
\end{equation*}
where $|\tilde A_n'|\sim 1$.  In the following proof, we will abuse notation and still denote $\tilde A_i$ as $A_i$, but with $|A_n'(x_1)|\sim 1$. 
  
    \subsubsection{Trivial decoupling}
    Let $\delta\in (0,1)$ and $\eps\in (0,1)$ be given. We apply a trivial decoupling to each variable $x_2,\dots,x_{n}$ to $\phi$-decouple $[-1,1]^{n-1}$ into smaller cubes $Q$ of side length $\delta^\eps$. By a translation and renaming the polynomials $A_j$, $1\le j\le n$ if necessary, we can assume that $0$ is a vertex of $Q$, namely, $Q=[0,\delta^\eps]^{n-1}$. The cost of decoupling is $\delta^{-O(\eps)}$ which is acceptable. Thus it remains to $\phi$-decouple $[-1,1]\times Q$ into $(\phi,\delta)$-flat parallelograms given by the following:
    \begin{equation}\label{eqn:decoupling_parallelograms_trivial_partition}
    [u,u+\delta^{1/2}]\times \{(x_2,\dots,x_{n})\in Q:A_2'(u)x_2+\cdots A_n'(u)x_n\in [v,v+\delta^{1/2}]\},
\end{equation}
where $u,v$ range through the $\delta^{1/2}$-lattice on $[-1,1]$ and $[-C,C]$, respectively.

    Denote by $D(\delta)$ the smallest decoupling constant associated to the decoupling in the previous sentence. We will need to show that $D(\delta)\le C_\eps \delta^{-\eps}$.

    \subsubsection{Induction on scales}
    
    Define $N=\eps^{-1}$ (of course we may assume $\eps\in \N^{-1}$), and define
    \begin{equation}
        \delta_i:=\delta^{i\eps},\quad 1\le i\le N.
    \end{equation}
    We aim to prove a bootstrap inequality of the form
    \begin{equation}\label{eqn:bootstrap}
        D(\delta_{i+1})\le C'_\eps D(\delta_i)\delta^{-\eps^2}, \quad \forall 1\le i\le N,
    \end{equation}
    for some constant $C'_\eps$. Then the result follows from iterating the boostrap inequality $N$ times.

    Let $1\le i\le N-1$ and suppose that we have $\phi$-decoupled into $(\phi,\delta_i)$-flat parallelograms $R_i$. We need to further $\phi$-decouple each $R_i$ into $(\phi,\delta_{i+1})$-flat parallelograms $R_{i+1}$.

    \subsubsection{Approximation}
    Fix one $R_i$, and by translation and renaming the polynomials $A_j$, $1\le j\le n$ if necessary, we may assume that $0$ is a vertex of $R_i$. By the induction hypothesis (at scale $\delta_i$), we may write $R_i$ in the form of \eqref{eqn:decoupling_parallelograms_trivial_partition}:
    \begin{equation*}
        [0,\delta_i^{1/2}]\times \{(x_2,\dots,x_{n})\in Q:A_2'(0)x_2+\cdots A_n'(0)x_n\in [0,\delta_i^{1/2}]\}.
    \end{equation*}
    Note that for all $x_1\in [0,\delta_i^{1/2}]$ and all $1\le j\le n$, we have
    \begin{equation*}
        |A_j(x_1)-A_j(0)-A'_j(0)x_1|\lesssim \delta_i.
    \end{equation*}
    As a result, we can approximate $\phi$ by
    \begin{equation}
        \tilde \phi(x):=A_1(x_1)+(A_2(0)+A'_2(0)x_1)x_2+\cdots+(A_{n}(0)+A'_{n}(0)x_1)x_{n},
    \end{equation}
    and since we also have $(x_2,\dots,x_{n})\in [0,\delta^\eps]^{n-1}$, the error of this approximation is $O(\delta_i \delta^\eps)=O(\delta_{i+1})$. Thus, it now suffices to $\phi$-decouple $R_i$ into $(\tilde \phi,\delta_i)$-flat parallelograms.

    \subsubsection{Applying lower dimensional result}
    We now focus on $\tilde \phi(x)$. Removing affine terms of $\tilde \phi$, it is equivalent to considering
    \begin{equation}
        A_1(x_1)+A'_2(0)x_1x_2+\cdots+A_{n}'(0)x_1x_{n}.
    \end{equation}
    But by an affine transformation in the $(x_2,\dots,x_{n})$ coordinate space, the above just reduces to study decoupling for $A_1(x_1)+A_n'(0)x_1 x_n$, which is the case $n=2$. One can also check that the decoupled parallelograms are also in the form of \eqref{eqn:decoupling_parallelograms}. This finishes the proof of the bootstrap inequality \eqref{eqn:bootstrap} and hence Proposition \ref{prop:Pramanik_Seeger}.

\section{Uniform decoupling implies sublevel set decoupling}\label{sec:U_imply_S}

In this section, we prove Theorem \ref{thm:main_U_imply_S}. More precisely, we will prove a superficially stronger statement, namely, $U(n-1)$ implies $\tilde S(n)$ as in  Definition \ref{defn:sublevel_set_decoupling_tilde}.

\begin{thm}\label{thm:sublevel_set_decoupling}
Let $n\ge 2$ and assume that $U(n-1)$ holds. Then $\tilde S(n)$ holds.
\end{thm}
\begin{cor}\label{cor:U_tilde_equivalent}
    For $n\ge 1$, $U(n)\iff\tilde U(n)$, and $S(n)\iff\tilde S(n)$.
\end{cor}
\begin{proof}
    To prove $U(n)\iff \tilde U(n)$, it suffices to prove $U(n)\implies\tilde U(n)$. By Theorem \ref{thm:sublevel_set_decoupling}, $U(n)$ implies $\tilde S(n+1)$, which then implies $S(n+1)$. By Lemma \ref{lem:trivial_implications_tilde}, $\tilde U(n)$ holds. 

    To prove $S(n)\iff \tilde S(n)$, it suffices to prove $S(n)\implies\tilde S(n)$. The case $n=1$ has been solved in Proposition \ref{prop:delta_thickening}, so we may assume $n\ge 2$. But by Lemma \ref{lem:trivial_implications}, $S(n)$ implies $U(n-1)$, which implies $\tilde S(n)$ by Theorem \ref{thm:sublevel_set_decoupling}.
\end{proof}

The rest of this section is devoted to the proof of Theorem \ref{thm:sublevel_set_decoupling}.
\subsection{Preliminary reductions}
We first show that it suffices to reduce to the following lemma.
\begin{lem}\label{lem:Section5}
Let $n\ge 2$. Then $\tilde S(n-1)$ and $\tilde U(n-1)$ together imply $S(n)$.  
\end{lem}
\begin{proof}[Proof of Theorem \ref{thm:sublevel_set_decoupling} assuming Lemma \ref{lem:Section5}]
The proof is by induction on $n$. The base case $n=2$ follows from Corollary \ref{cor:S(2)}. Now take $n\ge 3$, and assume that $U(n-2)$ implies $\tilde S(n-1)$. We would like to prove that $U(n-1)$ implies $\tilde S(n)$. 

Assume that $U(n-1)$ holds, then $U(n-2)$ holds by Lemma \ref{lem:monotonicity}. Then by the induction hypothesis, $\tilde S(n-1)$ holds, so $S(n-1)$ holds, and so $\tilde U(n-2)$ holds by Lemma \ref{lem:trivial_implications_tilde}. By Proposition \ref{prop:tilde_U(n)}, we also have $\tilde U(n-1)$. Since we have both $\tilde S(n-1)$ and $\tilde U(n-1)$, apply Lemma \ref{lem:Section5} gives $S(n)$. But we also have $\tilde S(n-1)$, so applying Proposition \ref{prop:tilde_S(n)}, we have $\tilde S(n)$.
\end{proof}

For the rest of this section, we prove Lemma \ref{lem:Section5}. Our idea is to first prove special cases, in a similar way to \cite[Section 4]{LiYang2023}.

\fbox{Notation} For $x=(x_1,\dots,x_n)$, we write $x'=(x_1,\dots,x_{n-1})$. We also abbreviate $\partial_{x_j}f$ to $\partial_j f$ or simply $f_j$, $1\le j\le n$. We first record the following simple observation.

\begin{prop}\label{prop:AJM}
Assume that $\tilde U(n-1)$ holds. Then the conclusion of $S(n)$ holds in the special case when $\phi(x)=x_n-P(x')$, where $P\in\mathcal P_{n-1,d}$.
\end{prop}

\subsection{Case of pseudo-polynomials}

Our next step is to prove decoupling in the following setup. 
\begin{defn}[Pseudo-polynomials]\label{defn:pseudo}
    Suppose that $\phi\in \mathcal P_{n,d}$ is such that $|\partial_{n} \phi|\gtrsim_{n,d} 1$ on $[-1,1]^n$ and that the equation $\phi(x)=0$ has a unique solution $x_n=\psi(x')$ on the whole of $x'\in [-1,1]^{n-1}$. We call $\psi$ a pseudo-polynomial of degree $d$ (defined by $\phi(x',\psi(x'))=0$).
\end{defn}
One can check that $\psi$ has bounded derivatives of all orders. Also, to decouple the sublevel set $|\phi|\le \delta$, it is essentially equivalent to $\psi$-decouple $[-1,1]^{n-1}$ into $(\psi,\delta)$-flat parallelograms.

\begin{prop}\label{prop:pseudo}
Assume that $\tilde U(n-1)$ holds. Let $\psi:[-1,1]^{n-1}\to \R$ be a pseudo-polynomial of degree $d$ defined by $\phi(x',\psi(x'))=0$. Then for $2\le p\le \frac{2(n+1)}{n-1}$, $\eps\in (0,1)$, $\delta\in (0,1)$, the unit square $[-1,1]^{n-1}$ can be $\psi$-$\ell^p(L^p)$ decoupled into boundedly overlapping $(\psi,\delta)$-flat parallelograms $T\in \mathcal T$ at scale $\delta$ at the cost of $O_\eps (\delta^{-\eps})$, such that each $T$ has width at least $\delta$. The implicit constants here do not depend on $\phi,\delta$.
\end{prop}

\begin{proof}
    The proof is by a Pramanik-Seeger type induction on scales, similar to that of Proposition \ref{prop:Pramanik_Seeger}. We provide the details below.

\fbox{Notation} All implicit constants below in this section are allowed to depend on $n,d$. All decoupling inequalities below are $\ell^p(L^p)$ where $2\le p\le \frac{2(n+1)}{n-1}$.

    \subsubsection{Preliminary reductions}
       
    We first apply the triangle and H\"older's inequalities to $\psi$-decouple $[-1,1]^{n-1}$ into $(\psi,\delta^\eps)$-flat squares $S_1$ of side length $\delta^\eps$, at cost $\delta^{-O(\eps)}$.

    \subsubsection{Induction on scales}
    We now set up the induction. Define $N=\eps^{-1}$ (of course we may assume $\eps\in \N^{-1}$), and define
    \begin{equation}
        \delta_i:=\delta^{i\eps},\quad 1\le i\le N.
    \end{equation}
    Let $i\ge 1$ and assume we have constructed a boundedly overlapping cover $\mathcal S_i$ of $S_1$ by $(\psi,\delta_i)$-flat parallelograms (the case $i=1$ is set up by $\mathcal S_1:=\{S_1\}$). We will construct a boundedly overlapping cover $\mathcal S_{i+1}$ of $S_1$ by $(\psi,\delta_{i+1})$-flat parallelograms. Also, denote by $D(\delta_i)$ the smallest decoupling constant corresponding to $\mathcal S_i$, and we will show that $D(\delta_i)\lesssim \delta_i^{\eps}$ by proving a bootstrap inequality of the form:
    \begin{equation}\label{eqn:bootstrap2}
        D(\delta_{i+1})\le C'_\eps D(\delta_i)\delta^{-\eps^2}, \quad \forall 1\le i\le N,
    \end{equation}
    for some constant $C'_\eps$. Then the result follows from iterating the boostrap inequality $N$ times.

    \subsubsection{Scale analysis}
    Suppose $i\ge 1$ and we have decoupled into some $S_i\in \mathcal S_i$. We want to further decouple $S_i$ at scale $\delta_{i+1}$. To this end, we first apply a few reductions.

    First, by a translation we may assume $0$ is a vertex of $S_i$. By another translation, we may assume $\psi(0)=0$. By a shear transform in the $z$ direction, we may assume $\nabla \psi(0)=0$. Using the fact that $\psi$ is $\delta_i$-flat over $S_i$, we have $|\psi|\le \delta_i$ over $S_i$.

    Write
    \begin{equation*}        \phi(x)=A_0(x')+A_1(x')x_n+A_2(x',x_n)x_n^2.
    \end{equation*}
    By a rescaling in the graphical direction, we may assume $A_1(0)=1$. Thus
    \begin{equation}\label{eqn:Feb_10}
        |A_1(x')-1|\lesssim \delta^\eps,\quad \forall x'\in S_1.
    \end{equation}
    
    \subsubsection{Approximation}
    We now approximate $\psi(x')$ by $-A_0(x')$ at scale $O(\delta_{i+1})$ on $[-1,1]^{n-1}$. Indeed, using $\phi(x',\psi(x'))=0$, \eqref{eqn:Feb_10} and $|\psi|\le \delta_i$, we compute that
    \begin{align*}
        \psi+A_0
        &=A_1\psi+(1-A_1)\psi+A_0
        =(1-A_1)\psi+O(\delta_i^2)\\
        &=O(\delta^\eps\delta_i)+O(\delta_i^2)        =O(\delta_{i+1}).
    \end{align*}
    Therefore, to decouple $S_i$ at scale $\delta_{i+1}$ for $\psi$, we equivalently decouple $S_i$ at scale $\delta_{i+1}$ for $A_0$, which is an element of $\mathcal P_{n-1,d}$ (after a trivial decoupling and rescaling if necessary). Applying Proposition \ref{prop:AJM} finishes the proof of the boostrap inequality \eqref{eqn:bootstrap2} and hence Proposition \ref{prop:pseudo}. Note that there are no rescaling in the coordinate space ($x_1,\dots, x_n$) involved, so the width lower bound $\delta$ given by Proposition \ref{prop:AJM} is retained.
\end{proof}

Proposition \ref{prop:pseudo} has the following corollary, which says that the conclusion of $S(n)$ holds when one partial derivative of $\phi$ is $\gtrsim 1$ on $[-1,1]^n$.

\begin{cor}\label{cor:sublevel_set_Pzsim1}
Assume that $\tilde U(n-1)$ holds. Suppose that $\phi\in \mathcal P_{n,d}$ is such that $|\partial_n \phi|\gtrsim 1$ on $[-1,1]^n$. Then for $2\le p\le \frac{2(n+1)}{n-1}$, $\eps\in (0,1)$, $\delta\in (0,1)$, the sublevel set 
\begin{equation*}
    \{x\in [-1,1]^n:|\phi(x)|\le \delta\}
\end{equation*}
can be $\ell^p(L^p)$-decoupled into boundedly overlapping parallelograms $T\in \mathcal T$ with width at least $\delta$ at the cost of $O_\eps (\delta^{-\eps})$, on each of which $|\phi|\lesssim \delta$. The implicit constants here do not depend on $\phi,\delta$.
\end{cor}
\begin{proof}
Since $|\partial_n \phi|\gtrsim 1$ over $[-1,1]^n$, there exists some $c_1(n,d)>0$ such that $|\partial_n \phi|\gtrsim 1$ over $[-1-c_1,1+c_1]^n$. If there exists $x\in [-1,1]^n$ such that $\phi(x)=0$, then by the implicit function theorem, there must exist a square $U\sub [-1-c_1,1+c_1]^{n-1}$ of length $\sim 1$ such that $\phi(x)=0$ defines an implicit function $x_n=\psi(x')$ for $x'\in U$. By cutting the $x'$ coordinate plane $[-1-c_1,1+c_1]^{n-1}$ into $O(1)$ many tiny squares, we can restrict ourselves to each such square. Then the result follows from applying Proposition \ref{prop:pseudo} following a harmless rescaling (see, for instance, \cite[Proposition 5.24]{LiYang2024}).
\end{proof}

\subsection{General case}\label{sec:general_case}
We now prove Lemma \ref{lem:Section5}. The idea is by induction on the degree $d$. For $d=0$ this is trivial. Assuming the result holds for $d-1\geq 0$, we now prove it for degree $d$, through the following steps.

\subsubsection{Dyadic decomposition}\label{sec:subsub_dyadic_decomposition}
We dyadically decompose $[-1,1]^n$ into subsets where $|\partial_1 \phi|\le \delta$ or $|\partial_1 \phi|\sim \sigma_1$ for some dyadic $\delta\le \sigma_1\lesssim 1$. We do the same for $\partial_j \phi$ for all $2\le j\le n$. It gives rise to $O(1)$ types of subsets according to the values of $\partial_j \phi$, namely, on each subset, for each $1\le j\le n$, either $|\partial_j \phi|\le \delta$, or is dyadically a constant $\sigma_j\ge \delta$. By trivial decoupling and losing $O(|\log \delta|^n)$, it suffices to decouple each such type of subset.

\subsubsection{Lower dimensional case}\label{sec:subsub_bivariate}
We first deal with an easier case where $|\partial_n \phi|\le \delta$. By symmetry, this solves the other cases where $|\partial_j \phi|\le \delta$ for some $j$.

Since $\partial_n \phi$ has degree at most $d-1$, we can apply the induction hypothesis to $\partial_n \phi$ to decouple $\{|\partial_n \phi|\le\delta\}$ into boundedly overlapping rectangles $T'_n$ of width at least $\delta$, on each of which we have $|\partial_n \phi|\lesssim \delta$. At the same time, we apply the induction hypothesis to other $\partial_j \phi$ to decouple $\{|\partial_j \phi|\le\delta\}$ (or $\{|\partial_j \phi|\sim\sigma_j\}$) into boundedly overlapping rectangles $T_j'$ with width at least $\sigma_j\ge \delta$. If $\cap_{j=1}^n T_j'\ne \varnothing$, then by a $\delta$-thickening, we may assume that $T':=\cap_{j=1}^n T_j'\ne \varnothing$ is still essentially an almost parallelogram with width at least $\delta$.

Thus, to decouple $T'\cap \{x:|\phi(x)|\le \delta\}$, by projection property, we decouple $\{x':|\phi(x',c)|\lesssim \delta\}$, where $c$ is the last coordinate of the centre of $T'$. But this follows from $\tilde S(n-1)$.

\subsubsection{Main case}
We now deal with the main case when $|\partial_j \phi|\sim \sigma_j\ge \delta$ for all $1\le j\le n$. We only need to deal with the case when $\partial_j \phi$ are all positive, since the other cases follow similarly. We may also assume $\sigma_n= \sigma:=\max\{\sigma_1,\dots,\sigma_{n}\}$ without loss of generality.

\underline{Step 1: Applying induction hypothesis.}

\smallskip

Since $\partial_n \phi$ has degree at most $d-1$, by the induction hypothesis, we may decouple $\{\partial_n \phi\sim \sigma\}$ into boundedly overlapping rectangles $T'_n$ of width at least $\sigma\ge\delta$, on each of which we still have $\partial_n \phi\sim \sigma$. (This requires a careful choice of the constant ratio in the dyadic decomposition in Section \ref{sec:subsub_dyadic_decomposition}; see \cite[Section 4.3.1]{LiYang2023}.) Similarly, for $1\le j\le n-1$, we may decouple $\{\partial_j \phi\sim \sigma_j\}$ into boundedly overlapping rectangles $T'_j$ of width at least $\sigma_j\ge\delta$, on each of which we still have $\partial_j \phi\sim \sigma_j$. Taking the intersection and applying a $\delta$-thickening trick if necessary, we may assume that $T':=\cap_{i=1}^n T'_j$ has width at least $\delta$.

\underline{Step 2: Rotation.}\smallskip

Fix a rectangle $T'$. By a translation we may assume $T'$ is centred at $0$. Let $\rho$ be an orthonormal transformation such that $\rho^{-1}(T')$ is axis-parallel. We have the following lemma.
\begin{lem}\label{lem:one_derivative_large}
    Let $\tilde \phi:=\phi\circ \rho$. Then there exists some $1\le j\le n$ such that $| \partial_j\tilde\phi|\sim \sigma$ on $\rho^{-1}(T')$.
\end{lem}
To prove the lemma, we first prove another elementary lemma.
\begin{lem}\label{lem:elementary}
    Let $n\in \N$. Then there exists some $c=(2n)^{-1/2}$ such that the following holds. Let $\vec \sigma=(\sigma_j)\in \R^n$, and assume $\vec v=(v_j)\in \R^n$ is such that $|v_j-\sigma_j|\le c|\sigma_j|$ for all $1\le j\le n$. Then for every orthogonal transformation $\rho:\R^n\to \R^n$, there exists $j=j(\rho,\vec \sigma)$ such that $|(\rho \vec v)_j|\ge c \norm{\vec \sigma}$. Here $\norm{\cdot}$ is the standard $\ell^2$ norm on $\R^n$.
\end{lem}
\begin{proof}
    Since $\norm{\rho\vec\sigma}=\norm{\vec \sigma}$, there exists some $j=j(\rho,\vec \sigma)$ such that $|(\rho \vec \sigma)_j|\ge 2c\norm{\vec \sigma}$. Also, using $|v_j-\sigma_j|\le c\sigma_j$ for each $j$, we compute that
    \begin{equation*}
        \norm{\rho(\vec \sigma-\vec v)}=\norm{\vec \sigma-\vec v}\le c\norm{\vec \sigma},
    \end{equation*}
    and in particular, for each $j$ we have
    \begin{equation*}
        |(\rho(\vec \sigma-\vec v))_j|\le c\norm{\vec \sigma}.
    \end{equation*}
    Thus by the triangle inequality, we have
    \begin{equation*}
        |(\rho \vec v)_j|\ge |(\rho \vec \sigma)_j|-|(\rho(\vec \sigma-\vec v))_j|\ge c\norm{\vec \sigma}.
    \end{equation*}
\end{proof}
We apply Lemma \ref{lem:elementary} with $\vec \sigma=(\sigma_1,\dots,\sigma_n)$, $\vec v=(\partial_1 \phi,\dots,\partial_n \phi)$ and $\rho^{-1}=\rho^T$ in place of $\rho$, which depends on $T'$ only. ($c$ can be chosen to be smaller than $(2n)^{-1/2}$ in the dyadic decomposition in Section \ref{sec:subsub_dyadic_decomposition}.) This implies Lemma \ref{lem:one_derivative_large} since $\norm{\vec\sigma}\sim \sigma$.

\underline{Step 3: Approximation and intermediate decoupling.}\smallskip

By Lemma \ref{lem:one_derivative_large}, we may assume without loss of generality that $T'$ is axis parallel and abuse notation $\tilde \phi=\phi$. For simplicity of notation, we only give the argument for the case where $\partial_n \phi\sim \sigma$ on $T'$.

Write $T'=R\times I$ where $I:=[-\eta,\eta]$, and the width $w$ of $T'$ is at least $\delta$ since the width of $T'$ is at least $\delta$. We approximate $\phi(x)$ by $\phi(x',0)$ with error $O(\eta\sigma)$. We have two cases. 

If $\eta\sigma\leq \delta$, then $\phi(x)$ can be approximated by $\phi(x',0)$ with an error of $O(\delta)$. Thus, the result follows from the same proof of the lower dimensional case in Section \ref{sec:subsub_bivariate}.

If $\eta\sigma>\delta$, we further consider two cases according to the sizes of $w$ and $\eta\sigma$, where 
\begin{equation*}
    w:=\text{width of $T'=R\times I$}.
\end{equation*}

\fbox{Case $w\ge \eta \sigma$}

We first apply a decoupling at the intermediate scale $\eta\sigma$. More precisely, by the assumption $\tilde S(n-1)$, we can decouple the set $\{x'\in R:|\phi(x',0)|\lesssim \eta\sigma\}$ into boundedly overlapping parallelograms $R'$ with width at least $\eta\delta$. Equivalently, we have decoupled $T'$ into rectangles $T''=R'\times [-\eta,\eta]$ with width at least $\eta\delta$, on each of which we have $|\phi(x)|\lesssim \eta\sigma$.

\underline{Step 4: Rescaling.}\smallskip

We need to further decouple $T''$. To this end, we define
\begin{equation*}
    \varphi(x)=c(\eta\sigma)^{-1} \phi(\lambda_{R'}(x'),\eta x_n),
\end{equation*}
which lies in $\mathcal P_{n,d}$ for a suitably small $c\sim 1$. Now recalling that $\partial_n \phi\sim \sigma$ on $T'$, we have $\varphi_n\sim 1$ on $[-1,1]^n$ by the rescaling.

\underline{Step 5: Applying Corollary \ref{cor:sublevel_set_Pzsim1}.}\smallskip

Now we apply Corollary \ref{cor:sublevel_set_Pzsim1} to $\varphi$ to decouple $\{|\varphi|< (\eta\sigma)^{-1}\delta\}$ into boundedly overlapping rectangles with width at least $(\eta\sigma)^{-1}\delta$, on each of which we have $|\varphi|\lesssim (\eta\sigma)^{-1}\delta$. Reversing the rescaling back to $T''$, we obtain a boundedly overlapping cover of $\{|\phi|<\delta\}$ by rectangles of width at least $\delta$, on each of which $|\phi|\lesssim \delta$.

\fbox{Case $w< \eta \sigma$}

This case is harder. We first note that $w$ is the width of $R'$ in this case, since $\eta\ge \eta \sigma$. We then apply an intermediate decoupling at scale $w$. More precisely, applying the fundamental theorem of algebra (more precisely, $\tilde S(1)$), we first trivially decouple the set $\{x_n:|\phi(c_R,x_n)|\le w\}$ into intervals $J$ with length $:=\eta'\ge w$, on each of which we have $|\phi(c_R,x_n)|\lesssim w$. Equivalently, we have decoupled $R\times [-\eta,\eta]$ into parallelograms of the form $R\times J$, on each of which we have $|\phi|\lesssim w$.

We then focus on one $R\times J$. Define
\begin{equation*}
    \psi(x',x_n)=w^{-1}\phi(\lambda_R x',\lambda_J x_n),
\end{equation*}
which has bounded coefficients. We then study decoupling of the sublevel set $\{x\in [-1,1]^n:|\psi(x)|<w^{-1}\delta\}$. 

We then apply the same argument as in Step 3 with the following choices:
\begin{equation*}
\begin{aligned}
    \text{New } &T' &=& \,\,[-1,1]^n\\
    \text{New } &\phi &=&\,\,\psi\\ 
    \text{New } &\sigma &=&\,\,\sigma \eta'w^{-1}\sim \psi_n\lesssim 1\\ 
    \text{New } &\delta &=&\,\,\delta w^{-1}\\
    \text{New } &w&=&\,\,1\\
    \text{New } &\eta&=&\,\,1. 
\end{aligned}    
\end{equation*}
Note that the product of the new $\eta$ and new $\sigma$ is equal to $\eta'w^{-1}\sim \psi_n\lesssim 1=\text{new }w$. So the rest of the argument in the case $w\ge \eta \sigma$ can be applied at the new scale $\delta w^{-1}$, this time giving parallelograms $\tilde R$ with width at least $\delta w^{-1}$. Lastly, rescaling back, we see that $\lambda_{R\times J}(\tilde R)$ has width at least $\delta$, since $R\times J$ has width $w$.

This closes the induction and thus finishes the proof of Theorem \ref{thm:sublevel_set_decoupling}.

\section{Decoupling for homogeneous polynomials}\label{sec:homo_highdim}
In this section, we prove Theorem \ref{thm:decoupling_homo} and Proposition \ref{prop:IZ(3)}. We will essentially follow the same outline of proof as in \cite[Section 12]{LiYang2025}, but generalised to higher dimensions.

\fbox{Notation} All implicit constants below in this section are allowed to depend on $n,d$. Unless otherwise specified, all decoupling inequalities below are $\ell^p(L^p)$ where $2\le p\le \frac{2(n+1)}{n-1}$.

\subsection{Preliminary reductions}\label{sec:reduction_homo_nonzero}
We induct on $n$. The base case $n=2$ is proved in \cite[Section 12]{LiYang2025}, since $H(3)$ holds. Let $n\ge 3$ and assume that $U(n-1)$ implies $H(n)$. We will prove that $U(n)$ implies $H(n+1)$.

Let $\phi(x'):=\phi(x,x_n)$ be homogeneous of degree $d\ge 1$, namely, $\phi(rx')=r^d \phi(x')$, $x'\in \R^{n+1}$. We may reparametrise the surface as
\begin{equation*}
    (x',\phi(x'))=(rs,r,r^d P(s)),\quad s\in \R^n,
\end{equation*}
where $P\in \mathcal P_{n,d}$. Then it suffices to prove a uniform decoupling result, with the cost of decoupling independent of the coefficients of $P$.

\subsection{A sublevel set decoupling}
We will need a higher dimensional analogue of the sublevel set decoupling in \cite[Theorem 12.1]{LiYang2025}, which decouples sublevel sets of the form
\begin{equation*}
    H_\delta:=\{(rs,r):r\in [1,2], s\in [-1,1]^n,|P(s)|<\delta\}.
\end{equation*}

\begin{thm}\label{thm:homo_sublevel_set}
Let $P\in \mathcal P_{n,d}$. For $\delta\ll_{n,d} 1$ and $\eps\in (0,1]$, there exists a boundedly overlapping cover $\mathcal S_\delta$ of $\{s\in [-1,1]^n:|P(s)|<\delta\}$ by parallelograms $S$, such that the following holds:
\begin{enumerate}    
    \item $|P|\lesssim \delta$ on each $S$.

    \item For every $2\le p\le \frac{2(n+1)}{n-1}$ and every $\eps\in (0,1]$, the set $H_\delta$ can be $\ell^p(L^p)$ decoupled into
    \begin{equation*}
        H_\delta(S):=\{(rs,r):r\in [1,2], (s_1,s_2)\in S\}
    \end{equation*}
    at cost $C_\eps \delta^{-\eps}$.
    \end{enumerate}
\end{thm}
To prove this theorem, we will follow a similar proof to that of Theorem \ref{thm:sublevel_set_decoupling}, and we just give the main ideas. 

\subsubsection{Case of implicit functions}
We first solve the case when $|\partial_n P|\sim 1$ over $[-1,1]^2$. In this case, the set $|P(s)|<\delta$ is essentially the $\delta$-neighbourhood of the graph of a pseudo-polynomial $s_n=g(\bar s):=g(s_1,\dots,s_{n-1})$, and so we just need to decouple the $\delta$-neighbourhood of
\begin{equation*}
    \{(r\bar s,rg(\bar s),r):r\in [1,2]:\bar s\in [-1,1]^{n-1}\}.
\end{equation*}
Swapping the last two coordinates and use the radial principle stated in \cite[Section 3]{LiYang2025} (or \cite[Section 2]{LiYang2024}), we can reduce the decoupling to $(t,g(t))$, which then follows from Proposition \ref{prop:pseudo}, since we have $\tilde U(n)$ (which is equivalent to $U(n)$, by Corollary \ref{cor:U_tilde_equivalent}).

\subsubsection{Induction on degree}
We may use essentially the same induction-on-degree argument as in Section \ref{sec:general_case} to reduce the general case to the case of implicit functions. This finishes the proof of Theorem \ref{thm:homo_sublevel_set}.

\subsection{Proof of Proposition \ref{prop:IZ(3)}}
By Lemma \ref{lem:zero_Hessian_paper}, we may assume $\psi(x,y)$ is of one of the following forms:
\begin{equation*}
\begin{aligned}
    (\text{Type I})\quad & \psi(x,y)=\eta(x)+\phi(y_1,y_2),\quad \phi\in \mathcal P_{2,d},\\
    (\text{Type II})\quad & \psi(x,y)=\eta(x)+A_1(y_1)+A_2(y_1)y_2+A_3(y_1)y_3,\quad A_1,A_2,A_3\in \mathcal P_{1,d}.
\end{aligned}    
\end{equation*}
The decoupling for Type I follows exactly from \cite[Proposition 13.1]{LiYang2025}. Decoupling for Type II can be proved in a similar way to that of Theorem \ref{thm:almost_affine} in the case $n=3$. We omit the details.

\subsection{Proof of decoupling for homogeneous polynomials}
Recall we have reduced to decoupling the surface
\begin{equation*}
    \{(rs,r,r^d P(s)):r\in [1,2],s\in [-1,1]^n\},
\end{equation*}
where $P\in \mathcal P_{n,d}$. By essentially the same argument in \cite[Section 12.3]{LiYang2025}, we can reduce to the decoupling to
\begin{equation*}
    \{(u,u^{\beta}t,u^{\beta}\psi(t):u\sim 1,|t|\le c\},
\end{equation*}
where $\beta=1/d$ and $\psi\in \mathcal P_{n,d}$. Applying the radial principle stated in \cite[Section 3]{LiYang2025}, this further reduces to
\begin{equation*}
    (u,t,u^\beta+\psi(t)),\quad u\sim 1,\quad  t\in [-c,c]^n.
\end{equation*}
This follows from applying the induction hypothesis and the degeneracy locating principle Theorem \ref{thm:degeneracy_locating_principle} again. For completeness, we give a little more details.

Let $\mathcal P$ consist of graphs over $[1,2]\times [-1,1]^n$ of $\phi(u,t):=\eta(u)+\psi(t)$, where $\eta$ is a smooth function with $\eta''\sim 1$ and $\psi\in \mathcal P_{n,d}$. Let $\mathcal R$ consist of all parallelograms of the form $I\times R$, where $I\sub [1,2]$ is an interval and $R\sub [-1,1]^n$ is a parallelogram. Let $\mathcal A$ consist of all affine bijections on $\R^{n+1}$ that can be factorised as $\lambda_I\times \lambda_R$, where $I\sub [1,2]$ is an interval and $R\sub [-1,1]^n$ is a parallelogram. Define $H\phi(u,t)=\det D^2 \psi(t)$.

We then need to verify the following ingredients of the degeneracy locating principle for $\mathcal P,\mathcal R,\mathcal A,H$:
\begin{itemize}
    \item $(\mathcal P,\mathcal R)$ is $\mathcal A$-rescaling invariant: they are easy to check, as they follow from a similar proof of Proposition \ref{prop:polynomial_rescaling_invariant}.
    
    \item $H$ is a degeneracy determinant: this follows from Corollary \ref{cor:hessian}.
    \item Trivial covering property: this is trivial, since every cube of side length $K^{-1}$ belongs to $\mathcal R$. 
    \item Sublevel set decoupling: this is just $S(n)$, which follows from the induction hypothesis $U(n)$.
    \item Totally degenerate decoupling: this is just the assumption $IZ(n)$ as in Definition \ref{defn:IZ(n)}.
    \item Nondegenerate decoupling: this follows from the induction hypothesis $U(n)$.
    \item Lower dimensional decoupling: The width of $I\times R$ is attained by $R$, so the lower dimensional decoupling follows from $\tilde U(n)\iff U(n)$ (we may approximate $u^{1/\beta}$ by its Taylor polynomial).
\end{itemize}

\bibliographystyle{alpha}
\bibliography{reference}

\end{document}